\documentclass[reqno]{amsart}
\usepackage{mathrsfs}
\usepackage{amssymb, amsmath, amsthm}
\usepackage{latexsym}
\usepackage{yfonts}
\usepackage{mathrsfs}
\usepackage{natbib}
\usepackage{graphicx,color}
\usepackage[pdftex,plainpages=false,colorlinks,hyperindex,bookmarksopen,linkcolor=blue,citecolor=red,urlcolor=red]{hyperref}
\usepackage{bm}
\DeclareMathAlphabet{\mathpzc}{OT1}{pzc}{m}{it}

\bibpunct{[}{]}{;}{n}{,}{,}

\newtheorem{te}{Theorem}[section]

\theoremstyle{definition}

\theoremstyle{os}
\newtheorem{os}[te]{Remark}

\theoremstyle{prop}

\theoremstyle{lem}

\theoremstyle{coro}
\newtheorem{coro}[te]{Corollary}

\theoremstyle{conj}

\numberwithin{equation}{section}

\allowdisplaybreaks

\begin{document}

	\title[Stable processes at random time]{Space-time fractional equations and the related stable processes at random time}
	\author{Enzo Orsingher}
	\email{enzo.orsingher@uniroma1.it}
	\author{Bruno Toaldo}
	\email{bruno.toaldo@uniroma1.it}	
	\address{Department of Statistical Sciences, Sapienza University of Rome}
	\date{\today}
	\dedicatory{}

\begin{abstract}
In this paper we consider the general space-time fractional equation of the form $\sum_{j=1}^m \lambda_j \frac{\partial^{\nu_j}}{\partial t^{\nu_j}} w(x_1, \cdots, x_n ; t) = -c^2 \left( -\Delta \right)^\beta w(x_1, \cdots, x_n ; t)$, for $\nu_j \in \left( 0,1 \right]$, $\beta \in \left( 0,1 \right]$ with initial condition $w(x_1, \cdots, x_n ; 0)= \prod_{j=1}^n \delta (x_j)$. We show that the solution of the Cauchy problem above coincides with the probability density of the $n$-dimensional vector process $\bm{S}_n^{2\beta} \left( c^2 \mathpzc{L}^{\nu_1, \cdots, \nu_m} (t) \right)$, $t>0$, where $\bm{S}_n^{2\beta}$ is an isotropic stable process independent from $\mathpzc{L}^{\nu_1, \cdots, \nu_m}(t)$ which is the inverse of $\mathpzc{H}^{\nu_1, \cdots, \nu_m} (t) = \sum_{j=1}^m \lambda_j^{1/\nu_j} H^{\nu_j} (t)$, $t>0$, with $H^{\nu_j}(t)$ independent, positively-skewed stable r.v.'s of order $\nu_j$. The problem considered includes the fractional telegraph equation as a special case as well as the governing equation of stable processes. The composition $\bm{S}_n^{2\beta} \left( c^2 \mathpzc{L}^{\nu_1, \cdots, \nu_m} (t) \right)$, $t>0$, supplies a probabilistic representation for the solutions of the fractional equations above and coincides for $\beta = 1$ with the $n$-dimensional Brownian motion at the random time $\mathpzc{L}^{\nu_1, \cdots, \nu_m} (t)$, $t>0$. The iterated process $\mathfrak{L}^{\nu_1, \cdots, \nu_m}_r (t)$, $t>0$, inverse to $\mathfrak{H}^{\nu_1, \cdots, \nu_m}_r (t) =\sum_{j=1}^m \lambda_j^{1/\nu_j} \, _1H^{\nu_j} \left( \, _2H^{\nu_j} \left( \, _3H^{\nu_j} \left( \cdots \, _{r}H^{\nu_j} (t) \cdots \right) \right) \right)$, $t>0$, permits us to construct the process $\bm{S}_n^{2\beta} \left( c^2 \mathfrak{L}^{\nu_1, \cdots, \nu_m}_r (t) \right)$, $t>0$, the density of which solves a space-fractional equation of the form of the generalized fractional telegraph equation. For $r \to \infty$ and $\beta = 1$ we obtain a probability density, independent from $t$, which represents the multidimensional generalisation of the Gauss-Laplace law and solves the equation $\sum_{j=1}^m \lambda_j w(x_1, \cdots, x_n) = c^2 \sum_{j=1}^n \frac{\partial^2}{\partial x_j^2} w(x_1, \cdots, x_n)$. Our analysis represents a general framework of the interplay between fractional differential equations and composition of processes of which the iterated Brownian motion is a very particular case.
\end{abstract}

	\maketitle

\section{Introduction and preliminaries}

\subsection{Introduction}
The study of the relationships between fractional differential equations and stochastic processes has gained considerable popularity during the past three decades. In pioneering works simple time-fractional diffusion equations have been considered (see for example Fujita \cite{fujita}) and its connection with stable processes has been established (see Orsingher and Beghin \cite{orsann}; the reader can also consult Zolotarev \cite{Zol86} for details on stable laws). In such papers the authors have shown that the compositions of processes have distributions satisfying fractional equations of different form. The iterated Brownian motion $B_1 \left( \left| B_2(t) \right| \right)$, $t>0$, (with $B_1$ and $B_2$ independent Brownian motions) has distribution solving the fractional equation (see Allouba and Zheng \cite{allouba})
\begin{equation*}
\frac{\partial^{\frac{1}{2}}}{\partial t^{\frac{1}{2}}} u(x, t) \, = \, \frac{1}{2^{\frac{3}{2}}} \frac{\partial^2}{\partial x^2} u(x, t), \qquad x \in \mathbb{R}, t>0.
\end{equation*}
as well as the fourth-order equation (see DeBlassie \cite{debl})
\begin{equation*}
\frac{\partial}{\partial t} u(x, t) \, = \, \frac{1}{2^3} \frac{\partial^4}{\partial x^4} u(x, t) + \frac{1}{2\sqrt{2\pi t}} \frac{d^2}{dx^2} \delta(x), \qquad x \in \mathbb{R}, t>0.
\end{equation*}
It has been shown by different authors (see Benachour  et al. \cite{benac}) that the solution to the biquadratic heat-equation
\begin{equation*}
\frac{\partial}{\partial t} u(x, t) \, = \, -\frac{1}{2^3} \frac{\partial^4}{\partial x^4} u(x, t), \qquad x \in \mathbb{R}, t>0,
\end{equation*}
coincides with
\begin{equation*}
u(x, t) \, = \, \mathbb{E} \left\lbrace \frac{1}{\sqrt{2\pi \left| B(t) \right|}} \cos \left( \frac{x^2}{2 \left| B(t) \right| - \frac{\pi}{4} } \right) \right\rbrace
\end{equation*}
and appears as the distribution of the composition of the Fresnel pseudoprocess with an independent Brownian motion (see Orsingher and D'Ovidio \cite{vibrat}).

When the fractional telegraph equation
\begin{equation}
\left( \frac{\partial^{2\nu}}{\partial t^{2\nu}} + 2\lambda \frac{\partial^\nu}{\partial t^\nu} \right) u(x,t) \, = \, c^2 \frac{\partial^2}{\partial x^2} u(x, t), \qquad x \in \mathbb{R}, t>0,
\label{T(B)}
\end{equation}
for $\nu \in \left( 0,1 \right]$, $\lambda>0$, $c>0$, is considered, the solution of problem \eqref{T(B)} for $\nu = \frac{1}{2}$ has been proved to coincide with the distribution of $T \left( \left| B(t) \right| \right)$, $t>0$, where $T(t)$, $t>0$, is a telegraph process independent from the Brownian motion $B(t)$, $t>0$ (see Orsingher and Beghin \cite{ptrf}). From the analytical point of view, equations similar to \eqref{T(B)} have been studied in the form
\begin{equation*}
\frac{\partial^\alpha}{\partial t^\alpha} u(x, t) + a \frac{\partial^\beta}{\partial t^\beta} u(x, t) \, = \, c^2 \frac{\partial^\gamma}{\partial x^\gamma} u(x, t) + \xi^2 u(x, t) + \varphi (x, t), \quad x \in \mathbb{R}, t>0,
\end{equation*}
for $\alpha \in \left[ 0,1 \right]$, $\beta \in \left[ 0,1 \right]$, $\gamma > 0$, by Saxena et al. \cite{saxena}. These authors have provided the Fourier transform  of solutions of fractional equations of the form
\begin{equation*}
a_1 \frac{\partial^{\alpha_1}}{\partial t^{\alpha_1}} u(x, t) + \cdots + a_{n+1} \frac{\partial^{\alpha_{n+1}}}{\partial t^{\alpha_{n+1}}} u(x, t) \, = \, c^2 \frac{\partial^\beta}{\partial x^\beta} u(x, t) + \xi^2 u(x, t) + \varphi (x, t) 
\end{equation*}
for $\alpha_1, \cdots, \alpha_{n+1} \in \left( 0,1 \right)$ and $\beta >0$, (see \cite{saxenaarx}) in terms of generalized Mittag-Leffler functions (but no probabilistc interpretation has been given to these solutions). Telegraph equations emerge in electrodynamics, in the study of damped vibrations, in the analysis of the telegraph process. Its multidimensional version appears in studying vibrations of membranes and other structures subject to friction. Equations with many fractional derivatives emerge in the study of anomalous diffusions as pointed out by \cite{saxena,saxenaarx}.

The symmetric stable laws have probability density satisfying the space-fractional equation
\begin{equation*}
\frac{\partial }{\partial t} u(x, t) \, = \, \frac{\partial^\nu}{\partial |x|^\nu} u(x, t), \qquad x \in \mathbb{R}, t>0,
\end{equation*}
where $\frac{\partial^\nu}{\partial |x|^\nu}$ is the Riesz fractional derivative. For asymmetric stable laws the connection with fractional equations has been established by Feller \cite{Feller52}. The connection between fractional telegraph equations and stable laws has been established in a recent paper by D'Ovidio et al. \cite{toaldo}, in which the authors considered the multidimensional space-fractional extension of \eqref{T(B)}
\begin{align}
\left( \frac{\partial^{2\nu}}{\partial t^{2\nu}} +2\lambda \frac{\partial^\nu}{\partial t^\nu} \right) u(\bm{x}, t) \, = \, -c^2 \left( -\Delta \right)^\beta u(\bm{x},t), \qquad \bm{x} \in \mathbb{R}^n, t>0, 
\label{dot}
 \end{align}
for $\nu \in \left( 0,\frac{1}{2} \right]$, $\beta \in \left( 0,1 \right]$.
The solution to \eqref{dot} subject to the initial condition $u(\bm{x},0)=\delta (\bm{x})$ is given by the law of the composition of the form $\bm{S}_n^{2\beta} \left( c^2 \mathpzc{L}^\nu (t) \right)$, $t>0$, where $\bm{S}_n^{2\beta} (t)$, $t>0$, is an $n$-dimensional isotropic stable vector process and 
\begin{equation*}
\mathpzc{L}^\nu (t) = \inf \left\lbrace s: H_1^{2\nu} (s) + (2\lambda)^{\frac{1}{\nu}} H_2^\nu (s) \geq t \right\rbrace
\end{equation*}
where $H_1^{2\nu}(t)$ and $H_2^\nu (t)$, $t>0$, are independent positively-skewed stable processes, with $\nu \in \left( 0,\frac{1}{2} \right]$. For $\beta =1$ the composition above takes the form of a Brownian motion at the delayed time $\mathpzc{L}^\nu (t)$, $t>0$. For $\nu = \frac{1}{2}$ and $n=1$ this establishes the fine distributional relationship
\begin{equation*}
T \left( \left| B(t) \right| \right) \, \stackrel{\textrm{law}}{=} \, B \left( c^2 \mathpzc{L}^{\frac{1}{2}} (t) \right), \qquad t>0,
\end{equation*}
see \cite{toaldo}.

In the present paper we consider the further generalization of the space-time fractional equation with an arbirtrary number of time-fractional derivatives
\begin{align}
\begin{cases}
\sum_{j=1}^m \lambda_j \frac{^C\partial^{\nu_j}}{\partial t^{\nu_j}} w_{\nu_1, \cdots, \nu_m}^\beta (\bm{x}, t) \, = \, -c^2  \left( -\Delta \right)^\beta w_{\nu_1, \cdots, \nu_m}^\beta (\bm{x}, t), \qquad \bm{x} \in \mathbb{R}^n, t>0, \\
w_{\nu_1, \cdots, \nu_m}^\beta (\bm{x}, 0) \, = \, \delta(\bm{x}),
\end{cases}
\label{nostro}
\end{align}
for $\nu_j \in \left( 0,1 \right]$, $\beta \in \left( 0,1 \right]$, $\lambda_j >0$, $j=1, \cdots, m $. The symbol $\frac{^C\partial^\nu}{\partial t^\nu}$ stands for the Dzerbayshan-Caputo fractional derivative which is defined as
\begin{equation*}
\frac{^C\partial^\nu}{\partial t^\nu} f(t) \, = \, \frac{1}{\Gamma \left( m-\nu \right)} \int_0^t \frac{\frac{d^m}{ds^m} f(s)}{(t-s)^{\nu +1 -m}} ds, \qquad m-1<\nu<m, m \in \mathbb{N},
\end{equation*}
for an absolutely continuous function $f$ (for fractional calculus the reader can consult Kilbas et al. \cite{kill}). Some basic facts on the fractional Laplacian $-\left( -\Delta \right)^\beta$, $\beta \in \left( 0,1 \right)$ are given in Section \ref{prelmult} below.
 We show that the solution to \eqref{nostro} is given by the law of the process $\bm{S}_n^{2\beta} \left( c^2 \mathpzc{L}^{\nu_1, \cdots, \nu_m} (t) \right)$, $t>0$, where
\begin{equation}
\mathpzc{L}^{\nu_1, \cdots, \nu_m} (t) \, = \, \inf \left\lbrace s >0: \mathpzc{H}^{\nu_1, \cdots, \nu_m} (s) > t \right\rbrace
\label{lpzc}
\end{equation}
and
\begin{equation}
\mathpzc{H}^{\nu_1, \cdots, \nu_m} (t) \, = \,  \sum_{j=1}^m \lambda_j^{\frac{1}{\nu_j}} H_j^{\nu_j} (t), \qquad t>0,
\label{hpzc}
\end{equation}
for $H_j^{\nu_j}$, $j=1, \cdots, m$, totally positively-skewed stable processes (stable subordinators), of order $\nu_j$. In other words we show that the solution of a general space-time fractional equation (which includes reaction-diffusion equations, telegraph equations, diffusion equations as very special cases) coincides with the distribution of a stable vector process taken at a random time $\mathpzc{L^\nu (t)}$, $t>0$, constructed as the inverse of the combination of independent stable subordinators. For the classical Laplacian ($\beta = 1$) we have that the solution to \eqref{nostro} is the distribution of a Brownian motion at time $\mathpzc{L}^\nu (t)$, $t>0$.

We also prove that the law of the processes \eqref{lpzc} and \eqref{hpzc} are solutions of fractional differential equations. In particular we show that
\begin{equation*}
\mathpzc{h}_{\nu_1, \cdots, \nu_m} (x, t) \, = \, \frac{\Pr \left\lbrace \mathpzc{H}^{\nu_1, \cdots, \nu_m} (t) \in dx \right\rbrace}{dx}, 
\end{equation*}
is the solution to the space-fractional problem for $\nu_j \in \left( 0,1 \right)$
\begin{align}
\begin{cases}
\frac{\partial}{\partial t} \mathpzc{h}_{\nu_1, \cdots, \nu_m} (x, t) \, = \, \sum_{j=1}^m \lambda_j \frac{\partial^{\nu_j}}{\partial x^{\nu_j}} \mathpzc{h}_{\nu_1, \cdots, \nu_m} (x, t), \qquad x >0, t>0 \\
\mathpzc{h}_{\nu_1, \cdots, \nu_m} (x, 0) \, = \, \delta(x), \\
\mathpzc{h}_{\nu_1, \cdots, \nu_m} (0, t) \, = \, 0,
\end{cases}
\label{110}
\end{align}
while the law of $\mathpzc{L}^{\nu_1, \cdots, \nu_m} (t)$ solves
\begin{align}
	\begin{cases}
 	\sum_{j=1}^m \lambda_j \frac{\partial^{\nu_j}}{\partial t^{\nu_j}} \mathpzc{l}_{\nu_1, \cdots, \nu_m} (x, t) \, = \, -\frac{\partial}{\partial x} \mathpzc{l}_{\nu_1, \cdots, \nu_m} (x, t), \qquad x >0, t>0, \\
 	\mathpzc{l}_{\nu_1, \cdots, \nu_m} (0, t) \, = \, \sum_{j=1}^m \lambda_j \frac{t^{\nu_j}}{\Gamma \left( 1-\nu_j \right)},
	\end{cases}
	\label{111}
\end{align}
for $\nu_j \in \left( 0,1 \right)$.
In \eqref{110} and \eqref{111} the fractional derivatives must be meant in the Riemann-Liouville sense that is, for an absolutely continuous function $f$, see \cite{kill},
\begin{equation*}
\frac{\partial^\nu}{\partial x^\nu} f(x) \, = \, \frac{1}{\Gamma \left( m-\nu \right)} \frac{d^m}{dx^m} \int_0^x \frac{f(s)}{(x-s)^{\nu +1-m}} ds, \qquad m-1 < \nu < m, m \in \mathbb{N}.
\end{equation*}

A section is devoted to the case of the fractional equation with two time derivatives of order $\alpha \in \left( 0,1 \right]$ and $\nu \in \left( 0,1 \right]$ with $\alpha \neq \nu$,
\begin{align}
\begin{cases}
\left( \frac{^C\partial^\alpha}{\partial t^\alpha} + 2\lambda \frac{^C\partial^\nu}{\partial t^\nu} \right) w_{\alpha, \nu}^\beta \left( \bm{x}, t \right) \, = \, -c^2 \left( -\Delta \right)^\beta w_{\alpha, \nu}^\beta \left( \bm{x}, t \right), \qquad \bm{x} \in \mathbb{R}^n, t>0, \\
w_{\alpha, \nu}^\beta \left( \bm{x}, 0 \right) \, = \, \delta (\bm{x}),
\end{cases}
\label{1122}
\end{align}
which takes a telegraph-type structure for $\alpha = k\nu$, $k \in \mathbb{N}$, $k\nu \leq 1$. The Fourier-Laplace transform of the solution of \eqref{1122} for $\alpha = k\nu$ reads
\begin{align}
\int_0^\infty dt \, e^{-\mu t} \int_{\mathbb{R}^n} d\bm{x} \, e^{i \bm{\xi} \cdot \bm{x}} w_{k\nu, \nu}^\beta \left( \bm{x}, t \right) \, = \, \frac{\mu^{k\nu-1}+2\lambda \mu^{\nu-1}}{\mu^{k\nu} + 2\lambda \mu^\nu + c^2 \left\| \bm{\xi} \right\|^{2\beta}},
\label{foulapl2der}
\end{align}
where $\left\| \cdot \right\|$ is the usual euclidean norm.
For $k=2$, $n=1$, $\beta =1$, we have the classical fractional telegraph equation studied in \cite{ptrf}. The Fourier transform of $w_{2\nu, \nu} (x, t)$ reads
\begin{align}
 \widehat{w}_{2\nu, \nu} (\xi, t) \, = \, \frac{1}{2} \left[ \left( 1+\frac{\lambda}{\sqrt{\lambda^2-c^2\xi^2}} \right) E_{\nu,1} \left( -\eta_1 t^\nu \right) + \left( 1-\frac{\lambda}{\sqrt{\lambda^2-c^2\xi^2}} \right) E_{\nu,1} \left( -\eta_2 t^\nu \right) \right]
\label{telegrafo}
\end{align}
where $\eta_1$ and $\eta_2$ are the solutions to $\mu^{2\nu} + 2\lambda \mu^\nu + c^2 \xi^2 = 0$ and
\begin{equation*}
E_{\psi,\vartheta} (z) = \sum_{k=0}^\infty \frac{z^k}{\Gamma \left( \psi k + \vartheta \right)}, \qquad \psi, \vartheta >0, z \in \mathbb{R},
\end{equation*}
is the two-parameter Mittag-Leffler function.
For $\nu =1$, \eqref{telegrafo} coincides with the characteristic function of the telegraph process. For $k=3$ and $\nu \leq \frac{1}{3}$ in \eqref{foulapl2der} we obtain explicitly the Fourier transform of the solutions in terms of Mittag-Leffler functions and the Cardano roots $A$, $B$ and $C$ of the third order algebraic equations $\mu^{3\nu} + 2\lambda \mu^\nu + c^2 \left\| \bm{\xi} \right\|^{2\beta} = 0$. For $k>3$ we can write
\begin{equation*}
\widetilde{\widehat{w}}_{k\nu, \nu}^\beta \left( \bm{\xi}, \mu \right) \, = \, \frac{\mu^{k\nu-1}+2\lambda \mu^{\nu-1}}{\mu^{k\nu}+2\lambda \mu^\nu + c^2 \left\| \bm{\xi} \right\|^{2\beta}} = \mu^{\nu-1} \prod_{i=1}^k \frac{\mu^{\nu-1}}{\mu^\nu -Z_i} + 2\lambda \mu^{\nu-1} \prod_{i=1}^k \frac{1}{\mu^\nu -Z_i}
\end{equation*}
but the explict evaluation of $Z_i$ is, in general, impossible.

In \cite{orsann} the $n$-times iterated Brownian motion
\begin{equation*}
\mathfrak{I}_n (t) \, = \, B_1 \left( \left| B_2 \left( \left| B_3 \cdots \left( \left| B_{n+1} (t) \right| \right) \cdots \right| \right)  \right| \right), \qquad t>0,
\end{equation*}
is considered and its connection with the fractional diffusion equation
\begin{equation*}
\frac{\partial^{\frac{1}{2^n}}}{\partial t^{\frac{1}{2^n}}} u(x, t) = 2^{\frac{1}{2^n}-2} \frac{\partial^2}{\partial x^2} u(x, t)
\end{equation*}
investigated. Here we consider first the $n$-times iterated positively-skewed stable process $_jH^{\nu_j}$ with weights $\lambda_j>0$, $j=1, \cdots, m$,
\begin{equation}
\mathfrak{H}_r^{\nu_1, \cdots, \nu_m} (t) = \sum_{j=1}^m \lambda_j^{\frac{1}{\nu_j}} \, _1H^{\nu_j} \left( \, _2H^{\nu_j} \left( \, _3H^{\nu_j} \left( \cdots \, _{r}H^{\nu_j} (t) \cdots \right) \right) \right), \qquad t>0.
\label{1200}
\end{equation}
We construct the inverse of the process \eqref{1200} as follows
\begin{equation*}
\mathfrak{L}_r^{\nu_1, \cdots, \nu_m} (t) \, = \, \inf \left\lbrace s>0: \mathfrak{H}_r^{\nu_1, \cdots, \nu_m} (s) \geq t \right\rbrace, \qquad t>0.
\end{equation*}
We show that the probability density of the composition
\begin{equation*}
\bm{B}_n \left( c^2 \mathfrak{L}_r^{\nu_1, \cdots, \nu_m} (t) \right), \qquad t>0,
\end{equation*}
where $\bm{B}_n$ represents an $n$-dimensional Brownian motion independent from $\mathfrak{L}_r^{\nu_1, \cdots, \nu_m}(t)$, is the solution to the Cauchy problem for $\nu_j \in \left( 0,1 \right]$, $r \in \mathbb{N}$,
\begin{align*}
\begin{cases}
\sum_{j=1}^m \lambda_j \frac{^C\partial^{\nu_j^r}}{\partial t^{\nu_j^r}} \mathfrak{w}^{r}_{\nu_1, \cdots, \nu_m} (\bm{x},t) = c^2 \Delta \mathfrak{w}^{  r}_{\nu_1, \cdots, \nu_m} (\bm{x},t), \qquad \bm{x} \in \mathbb{R}^n, t>0, \\
\mathfrak{w}^{ r}_{\nu_1, \cdots, \nu_m} (\bm{x},0) \, = \, \delta(\bm{x}).
\end{cases}
\end{align*}
We show that for the number $r$ of iterations tending to infinity
\begin{equation*}
\bm{B}_n \left( c^2 \mathfrak{L}_{r}^{\nu_1, \cdots, \nu_m} (t) \right) \begin{array}{c} \textrm{law} \\ \Longrightarrow \\ r \to \infty  \end{array} X_{m,n},
\end{equation*}
where $X_{m,n}$ is a r.v. independent from $t$ and possesses density equal to
\begin{equation}
\frac{\Pr \left\lbrace X_{m,n} \in d\bm{x} \right\rbrace}{d\bm{x}} \, = \, \frac{1}{(2\pi)^{\frac{n}{2}}} \left(  \frac{\sqrt{\sum_{j=1}^m \lambda_j}}{c} \right)^{\frac{n+2}{2}}   \left\| \bm{x} \right\|^{-\frac{n-2}{2}} K_{\frac{n-2}{2}} \left( \frac{\sqrt{\sum_{j=1}^m \lambda_j}}{c} \left\| \bm{x} \right\|  \right),
\label{introkappa}
\end{equation}
where $K_\nu (x)$ is the modified Bessel function. For $n=1$ the distribution \eqref{introkappa} becomes the Gauss-Laplace law
\begin{equation}
\mathfrak{w}_m \left( x \right) \, = \, \frac{\sqrt{\sum_{j=1}^m \lambda_j}}{2c} e^{-\frac{\sqrt{\sum_{j=1}^m \lambda_j}}{c} |x|}.
\label{126}
\end{equation}
Result \eqref{126} was obtained also in \cite{orsann} and by a different approach for $\lambda_1 = 1$, $\lambda_j = 0$ for $j \geq 2$, $c = \frac{1}{2}$, was derived by Turban \cite{turbo} as the limit of iterated random walks.

\subsection{Preliminaries}
\subsubsection{One dimensional stable laws}
\label{prelone}
Let us consider a stable process, say $S^{\nu} (t)$, $t>0$, $\nu \in \left( 0,2 \right]$, $\nu \neq 1$, for which, in general,
\begin{equation}
\mathbb{E}e^{i\xi S^{\nu} (t)} \, = \, e^{-\sigma \left| \xi \right|^\nu t \left( 1-i\theta \textrm{sign} (\xi) \tan \frac{\nu \pi}{2} \right) }
\label{stabili1dim}
\end{equation}
where $\theta \in \left[ -1,1 \right]$ is the skewness parameter and $\sigma = \cos \frac{\pi \nu }{2}$. In this paper we consider positively skewed processes ($\theta =1$, $0 < \nu < 1$) say $H^\nu (t)$, $t>0$, whose characteristic function writes
\begin{equation}
\widehat{h}_\nu \left( \xi, t  \right) \, = \, \mathbb{E}e^{i \xi H^\nu (t)} \, = \, e^{-t \left| \xi \right|^\nu \cos \frac{\pi \nu}{2} \left( 1-i \, \textrm{sign}(\xi) \tan \frac{\pi \nu}{2} \right)} \, = \, e^{-t \left( \left| \xi \right| e^{-\frac{i \pi}{2} \textrm{sign}(\xi)} \right)^\nu} \,  = \, e^{-t \left( -i\xi \right)^\nu }
\label{fouriersubordinatore}
\end{equation}
where we used the fact that $\left| \xi \right| e^{-i\frac{\pi}{2} \textrm{sign}(\xi)} = -i\xi$. The process $H^\nu (t)$, $t>0$, has the important property of having non-negative, stationary and independent increments, and thus is suitable to play the role of a random time. The law $h_\nu (x, t)$, $x \geq 0$, of $H^\nu (t)$, $t>0$, with Fourier transform $\widehat{h}_\nu (\xi, t)$ and Laplace transform
\begin{equation}
\widetilde{h}_\nu (\mu, t) \, = \, e^{-t \mu^\nu},
\label{laplace di subord}
\end{equation}
solves the fractional diffusion equation, for $\nu \in \left( 0,1 \right]$,
\begin{align*}
	\begin{cases}
	\left( \frac{\partial}{\partial t} + \frac{\partial^\nu}{\partial x^\nu} \right) h_\nu (x, t) \, = \, 0, \qquad x > 0, t>0, \\
	h_\nu (x, 0) \, = \, \delta(x), \\
	h_\nu (0, t) \, = \, 0,
	\end{cases}
	\label{}
\end{align*}
where the fractional derivatives are intended in the Riemann-Liouville sense.
We notice that the process given by the composition of $r \in \mathbb{N}$ independent stable subordinators of the same order $\nu \in \left( 0,1 \right)$, say $_1H^\nu \left( _2H^\nu \left( \cdots  _rH^\nu (t) \cdots \right)  \right)$, $t>0$ has law which reads
\begin{align}
& \frac{\Pr \left\lbrace _1H^\nu \left( _2H^\nu \left( \cdots  _rH^\nu (t) \cdots \right)  \right) \in dx \right\rbrace}{dx} \, = \, \notag \\
 = \, & \int_0^\infty ds_1 \, _1h_\nu (x, s_1) \int_0^\infty ds_2 \, _2h_\nu (s_1, s_2) \, \int_0^\infty ds_3 \, _3h_\nu  (s_2, s_3) \cdots \int_0^\infty ds_{r} \, _rh_\nu (s_r, t).
\label{legge di r sub}
\end{align}

In view of \eqref{fouriersubordinatore} and \eqref{laplace di subord} we can easily write the Laplace and Fourier transforms of \eqref{legge di r sub}. For example the Laplace transform reads
\begin{align}
	& \mathbb{E}e^{-\mu \; _1H^\nu \left( _2H^\nu \left( \cdots  _rH^\nu (t) \cdots \right)  \right)} \, = \, \int_0^\infty dx \, e^{-\mu x} \Pr \left\lbrace _1H^\nu \left( _2H^\nu \left( \cdots  _rH^\nu (t) \cdots \right)  \right) \in dx \right\rbrace  \notag \\
	 = \, &	\int_0^\infty ds_1 \, e^{-s_1 \mu^\nu} \int_0^\infty ds_2 \, _2h_\nu (s_1, s_2) \int_0^\infty ds_3 \, _3h_\nu (s_2, s_3) \cdots \int_0^\infty ds_r \, _rh_\nu (s_r, t) \notag \\
	 = \, & \int_0^\infty ds_2 e^{-s_2 \mu^{\nu^2}} \int_0^\infty ds_3 \, _3h_\nu (s_2, s_3) \cdots \int_0^\infty ds_r \, _rh_\nu (s_r, t) \, = \, e^{-t \mu^{\nu^r}}
	\label{557}
\end{align}
and therefore we have the following Fourier transform
\begin{equation}
\mathbb{E}e^{- i \xi  \, _1H^\nu \left( _2H^\nu \left( \cdots  _rH^\nu (t) \cdots \right)  \right)} \, = \, e^{-t \left( -i \xi \right)^{\nu^r}}
\label{17}
\end{equation}

\subsubsection{Multidimensional stable laws and fractional Laplacian}
\label{prelmult}
Let us consider the isotropic $n$-dimensional process $\bm{S}_n^{2\beta} (t)$, $t>0$ , $\beta \in \left( 0,1 \right]$, with density
\begin{equation}
v_\beta \left( \bm{x}, t \right) \, = \, \frac{1}{(2\pi)^n} \int_{\mathbb{R}^n} d\bm{\xi} \, e^{-i \bm{\xi} \cdot \bm{x}} \, e^{-t \left\| \bm{\xi} \right\|^{2\beta}}, \qquad \bm{x} \in \mathbb{R}^n, t>0,
\label{legge vettore iso}
\end{equation}
and therefore characteristic function
\begin{equation*}
\widehat{v}_\beta \left( \bm{\xi}, t \right) \, = \, \mathbb{E}e^{i\bm{\xi} \cdot \bm{S}_n^{2\beta} (t)} \, = \, e^{-t \left\| \bm{\xi} \right\|^{2\beta}},
\end{equation*}
where the symbol $\left\| \cdot \right\|$ stands for the usual Euclidean norm. The law \eqref{legge vettore iso} is the solution to the fractional Cauchy problem, for $\beta \in \left( 0,1 \right]$
\begin{align}
\begin{cases}
\left( \frac{\partial}{\partial t} + \left( - \Delta \right)^\beta \right) v_\beta \left( \bm{x}, t \right) \, = \, 0, \qquad \bm{x} \in \mathbb{R}^n, t>0, \\
v_\beta (\bm{x}, 0) \, = \, \delta(\bm{x}).
\end{cases}
\label{equazione governante vettore}
\end{align}
The fractional Laplacian appearing in \eqref{equazione governante vettore} has been considered by many authors (see, for example, Balakrishnan \cite{balakrishnan},  Bochner \cite{bochner}). The Bochner representation of the fractional Laplacian reads
\begin{equation*}
- \left( -\Delta \right)^\beta \, = \,  \frac{\sin \pi \beta}{\pi} \int_0^\infty d\lambda \, \lambda^{\beta -1} \left( \lambda - \Delta \right)^{-1} \, \Delta.
\end{equation*}
Equivalently, an alternative definition can be given in the space of the Fourier transforms, as
\begin{equation*}
- \left( -\Delta \right)^\beta u(\bm{x}) \, = \, \frac{1}{(2\pi)^n} \int_{\mathbb{R}^n} e^{-i \bm{x} \cdot \bm{\xi}} \left\| \bm{\xi} \right\|^{2\beta} \widehat{u} \left( \bm{\xi} \right) \, d\bm{\xi}
\end{equation*}
where
\begin{equation*}
\textrm{Dom} \left( -\Delta \right)^\beta \, = \, \left\lbrace u \in L^1_{\textrm{loc}} \left( \mathbb{R}^n \right) : \int_{\mathbb{R}^n} \left| \widehat{u} \left( \bm{\xi} \right) \right|^2 \left( 1+ \left\| \bm{\xi} \right\|^{2\beta} \right) d\bm{\xi} < \infty \right\rbrace.
\end{equation*}
In the one-dimensional case and for $0< 2\beta < 1$ we have that (see, for example, D'Ovidio et al. \cite{toaldo} for details on this point),
\begin{equation*}
 \left( - \frac{\partial^{2}}{\partial x^2} \right)^\beta u(x) \, = \, \frac{\partial^{2\beta}}{\partial |x|^{2\beta}} u(x),
\end{equation*}
where $\frac{\partial^{2\beta}}{\partial |x|^{2\beta}}$ is the Riesz operator usually defined as
\begin{equation*}
\frac{\partial^{2\beta}}{\partial |x|^{2\beta}} u(x) \, = \, -\frac{1}{2\cos \beta \pi} \frac{1}{\Gamma \left( 1-2\beta \right)} \frac{d}{dx} \int_{-\infty}^\infty \frac{u(z)}{|x-z|^{2\beta}} \, dz
\end{equation*}
and for which the Fourier transform becomes
\begin{equation*}
\mathcal{F} \left[ \frac{\partial^{2\beta}}{\partial |x|^{2\beta}} u(x) \right] \, = \, - \left| \xi \right|^{2\beta} \widehat{u} \left( \xi \right).
\end{equation*}

\section{Generalized fractional equations}
\subsection{Linear combination of stable processes}
In this section we start by considering processes of the form
\begin{equation}
\mathpzc{H}^{\nu_1, \cdots, \nu_m} (t) \, = \, \sum_{j=1}^m \lambda_j^{\frac{1}{\nu_j}} H_j^{\nu_j} (t), \qquad t>0, \nu_j \in \left( 0,1 \right), j = 1, \cdots, m,
\label{H generico}
\end{equation}
where $ H_j^{\nu_j} (t) $, $t>0$, are independent stable subordinators of order $\nu_j \in \left( 0,1 \right)$ introduced in section \ref{prelone}. Furthermore we will deal with the inverse process of $\mathpzc{H}^{\nu_1, \cdots, \nu_m}$, say $\mathpzc{L}^{\nu_1, \cdots, \nu_m}$, which can be defined as the hitting time of $\mathpzc{H}^{\nu_1, \cdots, \nu_m}$ as
\begin{equation}
\mathpzc{L}^{\nu_1, \cdots, \nu_m} (t) \, = \,  \inf \left\lbrace s>0: \mathpzc{H}^{\nu_1, \cdots, \nu_m} (s) = \sum_{j=1}^m \lambda_j^{\frac{1}{\nu_j}} H_j^{\nu_j} (s) \geq t \right\rbrace, \qquad t>0.
\label{L come hitting}
\end{equation}
The definition \eqref{L come hitting} of the process $\mathpzc{L}^{\nu_1, \cdots, \nu_m}$ permits us to write
\begin{equation}
\Pr \left\lbrace \mathpzc{L}^{\nu_1, \cdots, \nu_m} (t) < x \right\rbrace \, = \, \Pr \left\lbrace \mathpzc{H}^{\nu_1, \cdots, \nu_m} (x) >t \right\rbrace.
\label{proprietà inverso}
\end{equation}
 We present the following two results.
\begin{te}
\label{teorema hedl}
We have that
\begin{enumerate}
\item[i)] The solution to the problem for $\nu_j \in \left( 0,1 \right)$, $j=1, \cdots, m$,
\begin{align}
\begin{cases}
\frac{\partial}{\partial t} \mathpzc{h}_{\nu_1, \cdots, \nu_m} (x, t) \, = \, -\sum_{j=1}^m \lambda_j \frac{\partial^{\nu_j}}{\partial x^{\nu_j}} \mathpzc{h}_{\nu_1, \cdots, \nu_m} (x, t), \qquad x > 0, t>0, \\
\mathpzc{h}_{\nu_1, \cdots, \nu_m} (x, 0) \, = \, \delta (x), \\
\mathpzc{h}_{\nu_1, \cdots, \nu_m} (0, t) \, = \, 0.
\end{cases}
\label{problema h}
\end{align}
is given by the density of the process $\mathpzc{H}^{\nu_1, \cdots, \nu_m} (t)$, $t>0$, defined in \eqref{H generico}.

\item[ii)] The solution to the problem for $\nu_j \in \left( 0, 1 \right)$, $j = 1, \cdots, m$,
\begin{align}
	\begin{cases}
	\sum_{j=1}^m \lambda_j  \frac{\partial^{\nu_j}}{\partial t^{\nu_j}} \mathpzc{l}_{\nu_1, \cdots, \nu_m} (x, t) \, = \, -\frac{\partial}{\partial x} \mathpzc{l}_{\nu_1, \cdots, \nu_m} (x, t), \qquad x > 0, t>0, \\
	\mathpzc{l}_{\nu_1, \cdots, \nu_m} (0, t) \, = \, \sum_{j=1}^m \lambda_j \frac{t^{-\nu_j}}{\Gamma \left( 1-\nu_j \right)},
	\end{cases}
	\label{problema L}
\end{align}
is given by the probability density of the process $\mathpzc{L}^{\nu_1, \cdots, \nu_m} (t)$, $t>0$, defined in \eqref{L come hitting}.

\end{enumerate}
The fractional derivatives appearing in \eqref{problema h} and \eqref{problema L} must be intended in the Riemann-Liouville sense.

\end{te}
\begin{proof}[Proof of i)] Since for the Riemann-Liouville fractional derivative we have that
\begin{equation}
\mathcal{F} \left[ \frac{\partial^\nu}{\partial x^\nu} u(x) \right] \left( \xi \right) \, = \, \left( -i \xi \right)^\nu \widehat{u}(x)
\label{fourier di rl}
\end{equation}
we can write the Fourier transform of the problem \eqref{problema h} as
\begin{align*}
\frac{\partial}{\partial t} \widehat{\mathpzc{h}}_{\nu_1, \cdots, \nu_m} \left( \xi, t \right) \, = \, & -\mathcal{F} \left[ \sum_{j=1}^m \lambda_j \frac{\partial^{\nu_j}}{\partial x^{\nu_j}} \mathpzc{h}_{\nu_1, \cdots, \nu_m} (x, t)  \right] \left( \xi \right) \, = \,  \sum_{j=1}^m \lambda_j \left( -i \xi \right)^{\nu_j} \widehat{\mathpzc{h}}_{\nu_1, \cdots, \nu_m} \left( \xi, t \right),
 \end{align*}
and therefore we have that
\begin{align}
\begin{cases}
\frac{\partial}{\partial t} \widehat{h}_{\nu_1, \cdots, \nu_m} (\xi, t) \, = \, \sum_{j=1}^m \lambda_j \left( -i \xi \right)^{\nu_j} \widehat{\mathpzc{h}}_{\nu_1, \cdots, \nu_m} \left( \xi, t \right) \\
\widehat{h}_{\nu_1, \cdots, \nu_m} (\xi, 0) \, = \, 1.
\end{cases}
\label{fourier del problema di h}
\end{align}
The Fourier transform of the law $\mathpzc{h}_{\nu_1, \cdots, \nu_m} (x, t)$ of the process \eqref{H generico} is written as
 \begin{align}
\mathbb{E}e^{i \xi \mathpzc{H}^{\nu_1, \cdots, \nu_m} (t)} \, = \, \mathbb{E}e^{i\xi \sum_{j=1}^m \lambda_j^{\frac{1}{\nu_j}} H_j^{\nu_j} (t)} \,  \stackrel{\eqref{fouriersubordinatore}}{=} \, e^{- t \sum_{j=1}^m \lambda_j \left( - i \xi \right)^{\nu_j}}.
\label{tantamount}
 \end{align}
for which
\begin{align*}
	\frac{\partial}{\partial t} \mathbb{E}e^{i \xi \mathpzc{H}^{\nu_1, \cdots, \nu_m} (t)} \, = \, \sum_{j=1}^m \lambda_j \left( - i \xi \right)^{\nu_j} e^{-t \sum_{j=1}^m \lambda_j \left( - i \xi \right)^{\nu_j}}.
	\label{}
\end{align*}
This is tantamount to saying that the Fourier transform of $\mathpzc{h}_{\nu_1, \cdots, \nu_m} (x, t)$ is the solution to the problem \eqref{fourier del problema di h} and thus $\mathpzc{h}_{\nu_1, \cdots, \nu_m} (x, t)$ is the solution to \eqref{problema h}.
\end{proof}
\begin{proof}[Proof of ii)] 
In this proof we will make use of the Laplace transform of the Riemann-Liouville fractional derivative which, in view of \eqref{fourier di rl}, can be written as
\begin{equation*}
\mathcal{L} \left[ \frac{\partial^\nu}{\partial t^\nu} u(t) \right] \left( \mu \right) \, = \, \mu^\nu \widetilde{u} \left( \mu \right).
\end{equation*} 
Taking the Laplace transform of \eqref{problema L} with respect to $t$ we get
\begin{align}
	\sum_{j=1}^m \lambda_j \mu^{\nu_j} \widetilde{\mathpzc{l}}_{\nu_1, \cdots, \nu_m} (x, \mu) \, = \, -\frac{\partial}{\partial x} \widetilde{\mathpzc{l}}_{\nu_1, \cdots, \nu_m} (x, \mu),
	\label{t laplace L}
\end{align}
and performing the $x$-Laplace transform of \eqref{t laplace L} we arrive at
\begin{equation}
\sum_{j=1}^m \lambda_j \mu^{\nu_j} \widetilde{\widetilde{\mathpzc{l}}}_{\nu_1, \cdots, \nu_m} (\gamma, \mu) \, = \, \widetilde{\mathpzc{l}}_{\nu_1, \cdots, \nu_m} (0, \mu) - \gamma \widetilde{\widetilde{\mathpzc{l}}}_{\nu_1, \cdots, \nu_m} (\gamma, \mu).
\label{appearing in}
\end{equation}
The boundary condition appearing in \eqref{appearing in} can be derived from \eqref{problema L} as
\begin{equation*}
\widetilde{\mathpzc{l}}_{\nu_1, \cdots, \nu_m} (0, \mu) \, = \, \int_0^\infty dt \, e^{-\mu t} \sum_{j=1}^m \lambda_j \frac{t^{-\nu_j}}{\Gamma \left( 1-\nu_j \right)} \, = \, \sum_{j=1}^m {\lambda_j} \mu^{\nu_j-1}
\end{equation*}
and thus from \eqref{appearing in} we have that
\begin{equation}
\widetilde{\widetilde{\mathpzc{l}}}_{\nu_1, \cdots, \nu_m} (\gamma, \mu) \, = \, \frac{\sum_{j=1}^m {\lambda_j} \mu^{\nu_j-1}}{\sum_{j=1}^m {\lambda_j} \mu^{\nu_j} + \gamma}.
\label{analitical L}
\end{equation}
Now we show that the Fourier-Laplace transform of the law of the process $\mathpzc{L}^{\nu_1, \cdots, \nu_m} (t)$, $t>0$, coincides with \eqref{analitical L}. By taking into account the property \eqref{proprietà inverso} of the law of $\mathpzc{L}^{\nu_1, \cdots, \nu_m}$, we obtain
\begin{align*}
	 \widetilde{\widetilde{\mathpzc{l}}}_{\nu_1, \cdots, \nu_m} (\gamma, \mu) \, = \, & \int_0^\infty dt \, e^{-\mu t} \int_0^\infty dx \, e^{-\gamma x}  \mathpzc{l}_{\nu_1, \cdots, \nu_m} (x, t)  \notag \\ 
	\, = & \int_0^\infty dt \, e^{-\mu t} \int_0^\infty dx \, e^{-\gamma x} \left[ -\frac{\partial}{\partial x} \int_0^t dz \, \mathpzc{h}_{\nu_1, \cdots, \nu_m} (z, x) \right] \notag \\
	 = \, & - \frac{1}{\mu} \int_0^\infty dx \, e^{-\gamma x} \left[ \frac{\partial}{\partial x} \widetilde{\mathpzc{h}}_{\nu_1, \cdots, \nu_m}(\mu, x) \right] \notag \\
	 = \, &   - \frac{1}{\mu} \int_0^\infty dx \, e^{-\gamma x} \left[ \frac{\partial}{\partial x} \mathbb{E}e^{-\mu \sum_{j=1}^m \lambda_j^{\frac{1}{\nu_j}} H_j^{\nu_j} (x)} \right] \notag \\
 		\stackrel{\eqref{laplace di subord}}{=} \, &  - \frac{1}{\mu} \int_0^\infty dx \, e^{-\gamma x} \left[ \frac{\partial}{\partial x} e^{-x \sum_{j=1}^m \lambda_j \mu^{\nu_j}} \right] \, = \,  \frac{\sum_{j=1}^m \lambda_j \mu^{\nu_j-1}}{\sum_{j=1}^m \lambda_j \mu^{\nu_j} + \gamma},
	\label{}
\end{align*}
which coincides with \eqref{analitical L}. The proof of Theorem \ref{teorema hedl} is thus concluded.
\end{proof}

\subsection{Generalized fractional telegraph-type equations}
In this section we study equations of the form
\begin{equation}
\sum_{j=1}^m \lambda_j \frac{^C\partial^{\nu_j}}{\partial t^{\nu_j}} w_{\nu_1, \cdots, \nu_m}^\beta (\bm{x}, t) \, = \, -c^2 \left( -\Delta \right)^\beta w_{\nu_1, \cdots, \nu_m}^\beta (\bm{x}, t), \qquad \bm{x} \in \mathbb{R}^n, t>0,
\label{telegrafo generico}
\end{equation}
for $\nu_j \in \left( 0,1 \right]$, $j=1, \cdots, m$, $\beta \in \left( 0,1 \right]$, $c>0$, $\lambda >0$. The symbol $\frac{^C\partial^\nu}{\partial t^\nu}$ stands for the Dzerbayshan-Caputo fractional derivative. Equation \eqref{telegrafo generico} generalizes the telegraph equation in that an arbitrary number $m$ of time-fractional derivatives appears and the $n$-dimensional fractional Laplacian governs the space fluctuations. Concerning the equation \eqref{telegrafo generico} we present the following result.
\begin{te}
\label{teorema principe}
The solution to the problem for $\nu_j \in \left( 0, 1 \right]$, $j = 1, \cdots, m$, $\beta \in \left( 0, 1 \right]$,
\begin{align}
\begin{cases}
\sum_{j=1}^m \lambda_j \frac{^C\partial^{\nu_j}}{\partial t^{\nu_j}} w_{\nu_1, \cdots, \nu_m}^\beta (\bm{x}, t) \,  = \, -c^2 \left( -\Delta \right)^\beta w_{\nu_1, \cdots, \nu_m}^\beta (\bm{x}, t), \qquad x \in \mathbb{R}^n, t>0, \\
w_{\nu_1, \cdots, \nu_m}^\beta (\bm{x}, 0 ) \, = \, \delta \left( \bm{x} \right).
\end{cases}
\label{+generico}
\end{align}
is given by the law of the process
\begin{equation}
\bm{W}^{\nu_1, \cdots, \nu_m}_n (t) \, = \, \bm{S}_n^{2\beta} \left( c^2 \mathpzc{L}^{\nu_1, \cdots, \nu_m} (t) \right), \qquad t>0,
\label{processo composto}
\end{equation}
where $\bm{S}_n^{2\beta}$ is the isotropic vector process dealt with in section \ref{prelmult} and $\mathpzc{L}^{\nu_1, \cdots, \nu_m} (t)$, $t>0$ is the process defined in \eqref{L come hitting}.
\end{te}
\begin{proof}
Since for the Dzerbayshan-Caputo fractional derivative we have that,
\begin{equation*}
\mathcal{L} \left[ \frac{^C\partial^\nu}{\partial t^\nu} u(t) \right] \left( \mu \right) \, = \, \mu^\nu \widetilde{u}(\mu) - \mu^{\nu-1} u(0), \qquad \nu \in \left( 0,1 \right),
\end{equation*}
we can write the Laplace transform of \eqref{+generico} as
\begin{equation*}
\sum_{j=1}^m \lambda_j \mu^{\nu_j} \widetilde{w}_{\nu_1, \cdots, \nu_m}^\beta \left( x, \mu \right) - \sum_{j=1}^m \lambda_j \mu^{\nu_j-1} \delta(\bm{x}) \, = \, -c^2 \left( -\Delta \right)^\beta \widetilde{w}_{\nu_1, \cdots, \nu_m}^\beta \left( \bm{x}, \mu \right).
\end{equation*}
The Fourier-Laplace transform of \eqref{+generico} is therefore written as
\begin{equation}
\widehat{\widetilde{w}}_{\nu_1, \cdots, \nu_m}^\beta \left( \bm{\xi}, \mu \right) \, = \, \frac{\sum_{j=1}^m \lambda_j \mu^{\nu_j -1}}{\sum_{j=1}^m \lambda_j \mu^{\nu_j} + c^2 \left\| \bm{\xi} \right\|^{2\beta}}.
\label{f-l +generico}
\end{equation}
By considering \eqref{proprietà inverso} we can derive the Fourier-Laplace transform of the process \eqref{processo composto}. We have that
\begin{align*}
	\widehat{\widetilde{w}}_{\nu_1, \cdots, \nu_m}^\beta \left( \bm{\xi}, \mu \right) \, = & \, \int_{\mathbb{R}^n} d\bm{x} \, e^{i \bm{\xi} \cdot \bm{x}} \int_0^\infty dt \, e^{-\mu t} \int_0^\infty ds \, v_\beta \left( \bm{x}, c^2s \right) \, \mathpzc{l}_{\nu_1, \cdots, \nu_m} \left( s, t \right) \notag \\
	= \, & \int_0^\infty ds \, e^{-sc^2 \left\| \bm{\xi} \right\|^{2\beta}} \int_0^\infty dt \, e^{-\mu t} \left[- \frac{\partial}{\partial s} \int_0^t \mathpzc{h}_{\nu_1, \cdots, \nu_m} (z, s) \, dz \right]  \notag \\
	= \, & - \frac{1}{\mu}\int_0^\infty ds \, e^{-s c^2 \left\| \bm{\xi} \right\|^{2\beta}} \left(  \frac{\partial}{\partial s} \widetilde{\mathpzc{h}}_{\nu_1, \cdots, \nu_m} (\mu, s) \right) \notag \\
= \, & - \frac{1}{\mu}\int_0^\infty ds \, e^{-s c^2 \left\| \bm{\xi} \right\|^{2\beta}} \left(  \frac{\partial}{\partial s} \mathbb{E}e^{-\mu \sum_{j=1}^m \lambda_j H_j^{\nu_j} (s)} \right) \notag \\
\stackrel{\eqref{laplace di subord}}{=} \, & - \frac{1}{\mu}\int_0^\infty ds \, e^{-s c^2 \left\| \bm{\xi} \right\|^{2\beta}} \left(  \frac{\partial}{\partial s} e^{-s \sum_{j=1}^m \lambda_j \mu^{\nu_j}} \right) \notag \\
= \, & \sum_{j=1}^m \lambda_j \mu^{\nu_j-1} \int_0^\infty ds \, e^{-s c^2 \left\| \bm{\xi} \right\|^{2\beta}-s \sum_{j=1}^m \lambda_j \mu^{\nu_j}} \, = \,  \frac{\sum_{j=1}^m \lambda_j \mu^{\nu_j-1}}{ \sum_{j=1}^m \lambda_j \mu^{\nu_j} + c^2 \left\| \bm{\xi} \right\|^{2\beta}} \notag \\
& = \, \eqref{f-l +generico}.
	\label{}
\end{align*}
Since the Fourier-Laplace transform of the problem \eqref{+generico} coincides with that of the law of the process $\bm{S}_n^{2\beta} \left( c^2 \mathpzc{L}^{\nu_1, \cdots, \nu_m} (t) \right)$, $t>0$, the proof is complete.
\end{proof}
	
\subsection{Telegraph-type equations with two time-fractional derivatives}
When in the equation \eqref{telegrafo generico} only two time derivatives appear we can rewrite the problem, for $\alpha, \nu \in \left( 0,1 \right]$ as
\begin{align}
	\begin{cases}
	\left( \frac{^C\partial^\alpha}{\partial t^\alpha} + 2 \lambda \frac{^C\partial^\nu }{\partial t^\nu} \right) w_{\alpha,\nu}^\beta(\bm{x},t) \, = \, -c^2 \left(- \Delta \right)^\beta w_{\alpha,\nu}^\beta(\bm{x}, t), \qquad \bm{x} \in \mathbb{R}^n, t>0,  \\
	w_{\alpha, \nu}^\beta(\bm{x}, 0) \, = \, \delta(\bm{x}).
	\end{cases}
	\label{duederivate}
\end{align}
For $\alpha = 2\nu$, $\nu \in \left( 0, \frac{1}{2} \right]$ the reader can recongnize in \eqref{duederivate} the standard form of the classical fractional telegraph equation, investigated from a probabilistic point of view in \cite{ptrf} (for $n=1$ and $\beta =1$) and in \cite{toaldo} (for $n \in \mathbb{N}$ and $\beta \in \left( 0,1 \right)$). In view of Theorem \ref{teorema principe} it is not difficult to prove the following result.
	\begin{coro}
	\label{corollario}
	The solution of the fractional Cauchy problem \eqref{duederivate} is given by the probability density of the process
	\begin{equation}
	\bm{W}_n^{\alpha, \nu} (t)  \, \, = \, \bm{S}_n^{2\beta} \left( c^2 \mathpzc{L}^{\alpha, \nu} (t) \right), \qquad t>0,
	\label{process alfa nu}
	\end{equation}
	where
\begin{equation*}
\mathpzc{L}^{\alpha, \nu} (t) \, = \, \inf \left\lbrace s>0: \mathpzc{H}^{\alpha, \nu} (s) = H_1^\alpha + \left( 2\lambda \right)^{\frac{1}{\nu}} H_2^{\nu} (s) \geq  t \right\rbrace,
\end{equation*}
for $H_1^\alpha$ and $H_2^\alpha$ independent stable subordinators.
	\end{coro}
	\begin{proof}
The proof of this result can be carried out by repeating the arguments of Theorem \ref{teorema principe} and will not be reported here. It is sufficient to assume that $\lambda_1 = 1$, $\lambda_2 = 2\lambda$, $\lambda>0$ and $\lambda_j = 0$ for $j>2$.
\end{proof}

	\subsection{The case $\alpha = k \nu$}	
	Let us consider $\alpha = k \nu$, $\nu \in \left( 0, \frac{1}{k} \right]$, $k \in \mathbb{N}$, in \eqref{duederivate}. The problem becomes
	\begin{align}
	\begin{cases}
	\left( \frac{^C\partial^{k\nu}}{\partial t^{k\nu}} + 2 \lambda \frac{^C\partial^\nu }{\partial t^\nu} \right) w_{k\nu, \nu}^\beta(\bm{x},t) \, = \, -c^2 \left(- \Delta \right)^\beta w_{k\nu, \nu}^\beta(\bm{x}, t), \qquad \bm{x} \in \mathbb{R}^n, t>0,  \\
	w_{k\nu, \nu}^\beta(\bm{x}, 0) \, = \, \delta(\bm{x}).
	\end{cases}
	\label{nnu}
\end{align}
In view of Corollary \ref{corollario} the solution to \eqref{nnu} is given by the probability density of the process $\bm{S}_n^{2\beta} \left( c^2 \mathpzc{L}^{k\nu,\nu} (t) \right)$, $t>0$.
The Fourier-Laplace transform of $w_{k\nu, \nu}^\beta(\bm{x}, t)$ can be now written as
\begin{align*}
	& \widehat{\widetilde{w}}_{k\nu, \nu}^\beta  \left( \bm{\xi}, \mu \right)\, = \, \frac{\mu^{k \nu -1} + 2 \lambda \mu^{\nu -1}}{\mu^{n\nu} + 2\lambda \mu^\nu + c^2 \left\| \bm{\xi} \right\|^{2\beta}} \, = \, \mu^{n-1} \prod_{i=1}^k \frac{\mu^{\nu -1}}{\mu^\nu - Z_i} \, + \, 2\lambda \mu^{\nu -1} \prod_{i=1}^k \frac{1}{\mu^\nu - Z_i}
	\label{}
	\end{align*}
	where $Z_i$ are the roots of $ \mu^{k\nu} + 2\lambda \mu^\nu + c^2 \left\| \bm{\xi} \right\|^{2\beta}  =  0 $.
	
For $k=3$ we get
	\begin{align}
	\widehat{\widetilde{w}}_{3\nu, \nu}^\beta \left( \bm{\xi}, \mu \right) \, = \,  \frac{\mu^{3\nu -1}}{\mu^\nu - A} \, \frac{1}{\mu^\nu - B} \, \frac{1}{\mu^\nu - C}  +  2 \lambda \frac{\mu^{\nu -1}}{\mu^\nu - A} \,  \frac{1}{\mu^\nu - B} \, \frac{1}{\mu^\nu - C},
	\label{f-l 3}
	\end{align}
	where $A, \, B$ and $C$ are the solutions to $\mu^{3\nu} + 2\lambda \mu^\nu + c^2 \left\| \bm{\xi} \right\|^{2\beta}  =  0$.
Formula \eqref{f-l 3} can be rewritten as
\begin{align}
	\widehat{\widetilde{w}}_{3\nu, \nu}^\beta \left( \bm{\xi}, \mu \right)   \, = \, & \frac{ \left( \mu^{3\nu -1} + 2\lambda \mu^{\nu -1} \right)}{\mu^\nu - A}  \left[ \left( \frac{1}{\mu^\nu - B} - \frac{1}{\mu^\nu - C} \right) \frac{1}{B-C} \right]  \notag \\
	 = \, & \left( \mu^{3\nu -1} + 2\lambda \mu^{\nu -1} \right) \left[ \left( \frac{1}{\mu^\nu - A } - \frac{1}{ \mu^\nu - B} \right) \frac{1}{\left( A-B \right) \left( B-C \right)} \right. \notag \\
	 & - \left. \left( \frac{1}{\mu^\nu - A} - \frac{1}{\mu^\nu - C} \right)  \frac{1}{\left( A-C \right) \left( B-C \right)} \right] \notag \\
	 = \, & \left( \mu^{3\nu -1} + 2\lambda \mu^{\nu -1} \right) \left[ \frac{1}{\mu^{\nu}-A} \frac{1}{\left( B-A \right) \left( C-A \right)} + \frac{1}{\mu^\nu - B} \frac{1}{\left( A-B \right) \left( C-B \right)} \right. \notag \\
	 & + \left. \frac{1}{\mu^\nu -C} \frac{1}{\left( A-C \right) \left( B-C \right)} \right].
	\label{f-l 3 riscritta}
\end{align}
By considering now the relationship
\begin{equation*}
	\int_0^\infty e^{-\mu t} t^{(1-2\nu)-1} E_{\nu, 1-2\nu} \left( C t^\nu \right) dt \, = \, \frac{\mu^{\nu - (1-2\nu)}}{\mu^\nu - C}
	\label{}
\end{equation*}
we can invert \eqref{f-l 3 riscritta} with respect to $\mu$. Thus we can explicitly write the characteristic function of the process $\bm{S}_n^{2\beta} \left( c^2 \mathpzc{L}^{3\nu, \nu} (t) \right)$, $t>0$, as
\begin{align*}
\mathbb{E}e^{i\bm{\xi} \cdot \bm{S}_n^{2\beta} \left( c^2 \mathpzc{L}^{3\nu, \nu} (t) \right) } =  &  \frac{ t^{-2\nu} E_{\nu, 1-2\nu} \left( A t^\nu \right) + 2\lambda E_{\nu, 1} \left( A t^\nu \right) }{\left( B-A \right) \left( C-A \right) } + \frac{t^{-2\nu} E_{\nu, 1-\nu} \left( Bt^\nu \right) + 2\lambda E_{\nu,1} \left( Bt^\nu \right)}{\left( A-B \right) \left( C-B \right)} \notag \\
 & + \frac{t^{-2\nu} E_{\nu, 1-2\nu} \left( Ct^\nu \right) + 2\lambda E_{\nu,1} \left( Ct^\nu \right)}{\left( A-C \right) \left( B-C \right)}.
\end{align*}

\section{Multidimensional Gauss-Laplace distributions and infinite compositions}
In \cite{orsann} the authors have shown that the process 
\begin{equation*}
\mathfrak{I}_n (t) \, = \, B_1 \left( \left| B_2 \left( \left| B_3 \cdots \left( \left| B_{n+1} (t) \right| \right) \cdots \right| \right)  \right| \right), \qquad t>0,
\end{equation*}
converges in distribution for $n \to \infty$ to a Gauss-Laplace (or bilateral exponential) random variable independent from $t>0$. In this section we show that the process $\bm{B}_n \left( \mathfrak{L}_{r}^{\nu_1, \cdots, \nu_m} (t) \right)$, $t>0$, converges in distribution, for $r \to \infty$, to a multidimensional version of the Gauss-Laplace r.v. and its distribution solves the equation, for $\nu_j \in \left( 0,1 \right)$, $r \in \mathbb{N}$,
\begin{equation*}
\sum_{j=1}^m \lambda_j \frac{^C\partial^{\nu_j^r}}{\partial t^{\nu_j^r}} \mathfrak{w}^{\beta, r}_{\nu_1, \cdots, \nu_m} (\bm{x}, t) \, = \, c^2 \Delta \mathfrak{w}^{\beta, r}_{\nu_1, \cdots, \nu_m} (\bm{x}, t), \qquad \bm{x} \in \mathbb{R}^n, t>0.
\end{equation*}
The process $\mathfrak{L}_{r}^{\nu_1, \cdots, \nu_m} (t)$, $t>0$, is defined as
\begin{equation*}
\mathfrak{L}_{r}^{\nu_1, \cdots, \nu_m} (t) \, = \,  \inf \left\lbrace s>0: \mathfrak{H}_{r}^{\nu_1, \cdots, \nu_m} (s) \geq t \right\rbrace, \qquad t>0
\end{equation*}
where
\begin{align*}
\mathfrak{H}_{r}^{\nu_1, \cdots, \nu_m} (t) \, = \,  \sum_{j=1}^m \lambda_j^{\frac{1}{\nu_j}} \, _1H^{\nu_j} \left( \, _2H^{\nu_j} \left( \, _3H^{\nu_j} \left( \cdots \, _{r}H^{\nu_j} (t) \cdots \right) \right) \right), \qquad t>0.
\end{align*}
We start by presenting the following results.
\begin{coro} 
We have that
\begin{enumerate}
\item[i)] The solution to the problem for $\nu_j \in \left( 0 , 1 \right)$, $j=1, \cdots, m$, $r \in \mathbb{N}$,
\begin{align}
	\begin{cases}
	\frac{\partial}{\partial t} \, \mathfrak{h}^{r}_{\nu_1, \cdots, \nu_m} (x, t) = \, -\sum_{j=1}^m \lambda_j \frac{\partial^{\nu_j^r}}{\partial x^{\nu_j^r}} \mathfrak{h}^{r}_{\nu_1, \cdots, \nu_m} (x, t), \qquad x > 0, t>0, \\
	\mathfrak{h}^{r}_{\nu_1, \cdots, \nu_m} (x, 0) \, = \, \delta(x), \\
	\mathfrak{h}^{r}_{\nu_1, \cdots, \nu_m} (0, t) \, = \, 0,
	\end{cases}
	\label{hstrano}
\end{align}
is given by the law of the process
\begin{align*}
\mathfrak{H}_{r}^{\nu_1, \cdots, \nu_m} (t) \, = \,  \sum_{j=1}^m \lambda_j^{\frac{1}{\nu_j}} \, _1H^{\nu_j} \left( \, _2H^{\nu_j} \left( \, _3H^{\nu_j} \left( \cdots \, _{r}H^{\nu_j} (t) \cdots \right) \right) \right), \qquad t>0.
\end{align*}
\item[ii)] The solution to the problem for $\nu_j \in \left( 0, 1 \right)$, $j = \, \cdots, m$, $r \in \mathbb{N}$,
\begin{align}
	\begin{cases}
	\sum_{j=1}^m \lambda_j \frac{\partial^{\nu_j^r}}{\partial t^{\nu_j^r}} \mathfrak{l}^{r}_{\nu_1, \cdots, \nu_m} (x, t) \, = \, -\frac{\partial}{\partial x} \mathfrak{l}^{r}_{\nu_1, \cdots, \nu_m} (x, t), \qquad x > 0, t>0, \\
	\mathfrak{l}^{r}_{\nu_1, \cdots, \nu_m} (0, t) \, = \, \sum_{j=1}^m \lambda_j \frac{t^{\nu_j^r}}{\Gamma \left( 1- \nu_j^r \right)}
	\end{cases}
	\label{lstrano}
\end{align}
is given by the law of the process 
\begin{align}
\mathfrak{L}_{r}^{\nu_1, \cdots, \nu_m} (t) \, = \,  \inf \left\lbrace s>0: \mathfrak{H}_{r}^{\nu_1, \cdots, \nu_m} (s) \geq t \right\rbrace, \qquad t>0.
\label{lfrac inverso}
\end{align}
\end{enumerate}
\end{coro}
\begin{proof}[Proof of i)]
The proof is carried out in the same spirit of Theorem \ref{teorema hedl}, thus by considering the Fourier transform of \eqref{hstrano} we get
\begin{align}
\begin{cases}
\frac{\partial}{\partial t} \widehat{\mathfrak{h}}^{r}_{\nu_1, \cdots, \nu_m} (\xi, t) \, = \, - \sum_{j=1}^m \lambda_j \left( - i \xi \right)^{\nu_j^r} \widehat{\mathfrak{h}}^{r}_{\nu_1, \cdots, \nu_m} (\xi, t) \\
\widehat{\mathfrak{h}}^{r}_{\nu_1, \cdots, \nu_m} (\xi, 0) \, = \, 1.
\end{cases}
\label{hstranofourier}
\end{align}
The proof is completed by observing that the solution to \eqref{hstranofourier} is given by the Fourier transform of the law of the process $\mathfrak{H}_{r}^{\nu_1, \cdots, \nu_m} (t)$, $t>0$, which can be obtained by means of the calculation
\begin{align*}
\mathbb{E}e^{i\xi \mathfrak{H}_{r}^{\nu_1, \cdots, \nu_m} (t)} \,  = \, & \mathbb{E}e^{i\xi \sum_{j=1}^m \lambda_j^{\frac{1}{\nu_j}} \, _1H^{\nu_j} \left( \, _2H^{\nu_j} \left( \, _3H^{\nu_j} \left( \cdots \, _{r}H^{\nu_j} (t) \cdots \right) \right) \right)} \, \stackrel{\eqref{17}}{=} \, e^{-t \sum_{j=1}^m \lambda_j \left( -i \xi \right)^{\nu_j^r}},
\end{align*}
that is the solution to \eqref{hstranofourier}.
\end{proof}
\begin{proof}[Proof of ii)]
By considering the double Laplace transform of \eqref{lstrano} we have that
\begin{equation*}
\sum_{j=1}^m \lambda_j \mu^{\nu_j^r} \, \widetilde{\widetilde{\mathfrak{l}}} \, ^{r}_{\nu_1, \cdots, \nu_m} \left( \gamma, \mu \right) \, = \, \widetilde{\mathfrak{l}}^{r}_{\nu_1, \cdots, \nu_m} \left( 0, \mu \right) - \gamma \widetilde{\widetilde{\mathfrak{l}}} \, ^{r}_{\nu_1, \cdots, \nu_m} \left( \gamma, \mu \right),
\end{equation*}
where the boundary condition is given by
\begin{equation*}
\int_0^\infty dt \, e^{-\mu t} \sum_{j=1}^m \lambda_j \frac{t^{\nu_j^r}}{\Gamma \left( 1-{\nu_j^r} \right)} \, = \, \sum_{j=1}^m \lambda_j \mu^{\nu_j^r -1},
\end{equation*}
and thus
\begin{equation}
\widetilde{\widetilde{\mathfrak{l}}} \, ^{r}_{\nu_1, \cdots, \nu_m} \left( \gamma, \mu \right) \, = \, \frac{\sum_{j=1}^m \lambda_j \mu^{\nu_j^r-1}}{\sum_{j=1}^m \lambda_j \mu^{\nu_j^r} + \gamma}
\label{dllstrano}
\end{equation}
The definition \eqref{lfrac inverso} permits us to state that the processes $\mathfrak{L}_{r}^{\nu_1, \cdots, \nu_m} (t)$, $t>0$, and $\mathfrak{H}_{r}^{\nu_1, \cdots, \nu_m} (t)$, $t>0$, are related by the fact that
\begin{equation*}
\Pr \left\lbrace \mathfrak{L}_{r}^{\nu_1, \cdots, \nu_m} (t) < x \right\rbrace \, = \, \Pr \left\lbrace \mathfrak{H}_{r}^{\nu_1, \cdots, \nu_m} (x) > t \right\rbrace,
\end{equation*}
and thus we can perform manipulations similar to those of Theorem \ref{teorema hedl}. We have that the double Laplace transform of the law $\mathfrak{l}^{r}_{\nu_1, \cdots, \nu_m} (x, t)$ is then given by
\begin{align*}
	\widetilde{\widetilde{\mathfrak{l}}} \, ^{r}_{\nu_1, \cdots, \nu_m} (\gamma, \mu) \, = \, & \int_0^\infty dt \, e^{-\mu t} \int_0^\infty dx \, e^{-\gamma x} \left[ -\frac{\partial}{\partial x} \int_0^t dz \, \mathfrak{h}^{r}_{\nu_1, \cdots, \nu_m} (z, x) \right] \notag \\
	= \, & - \frac{1}{\mu} \int_0^\infty dx \, e^{-\gamma x} \left[ \frac{\partial}{\partial x} \widetilde{\mathfrak{h}}^{r}_{\nu_1, \cdots, \nu_m} (\mu, x) \right] \notag \\
	= \, & - \frac{1}{\mu} \int_0^\infty dx \, e^{-\gamma x} \left[ \frac{\partial}{\partial x} \mathbb{E}e^{-\mu \sum_{j=1}^m \lambda_j^{\frac{1}{\nu_j}} \, _1H^{\nu_j} \left( \, _2H^{\nu_j} \left( \, _3H^{\nu_j} \left( \cdots \, _{r}H^{\nu_j} (t) \cdots \right) \right) \right)} \right] \notag \\
	\stackrel{\eqref{557}}{=} \, & - \frac{1}{\mu} \int_0^\infty dx \, e^{-\gamma x} \left[ \frac{\partial}{\partial x} e^{-x \sum_{j=1}^m \lambda_j \mu^{\nu_j^r}} \right] \,= \,  \frac{\sum_{j=1}^m \lambda_j \mu^{\nu_j^r-1}}{\sum_{j=1}^m \lambda_j \mu^{\nu_j^r} + \gamma},
	\label{}
\end{align*}
and coincides with \eqref{dllstrano}.
\end{proof}

\begin{te}
The solution to the problem for $\nu_j \in \left( 0, 1 \right]$, $\beta \in \left( 0,1 \right]$, $j=1, \cdots, m$, $r \in \mathbb{N}$,
\begin{align}
\begin{cases}
\sum_{j=1}^m \lambda_j \frac{^C\partial^{\nu_j^r}}{\partial t^{\nu_j^r}} \mathfrak{w}_{\nu_1, \cdots, \nu_m}^{\beta, r} \left( \bm{x}, t \right) \, = \, - c^2 \left( - \Delta \right)^\beta \mathfrak{w}_{\nu_1, \cdots, \nu_m}^{\beta, r} \left( \bm{x}, t \right), \qquad \bm{x} \in \mathbb{R}^n, t>0, \\
\mathfrak{w}_{\nu_1, \cdots, \nu_m}^{\beta, r} \left( \bm{x}, 0 \right) \, = \, \delta \left( \bm{x} \right),
\end{cases}
\label{chepalle}
\end{align}
is given by the probability density of the process
\begin{equation}
\bm{\mathfrak{W}}_n \left( t \right) \, = \, \bm{S}_{n}^{2\beta} \left( c^2 \mathfrak{L}_r^{\nu_1, \cdots, \nu_m} (t) \right), \qquad t >0.
\label{processo figo}
\end{equation}
where the process $\mathfrak{L}_r^{\nu_1, \cdots, \nu_m} (t)$, $t>0$, is defined in \eqref{lfrac inverso}. For $\beta = 1$, the process \eqref{processo figo} becomes the subordinated Brownian motion $\bm{B}_n \left( c^2 \mathfrak{L}_r^{\nu_1, \cdots, \nu_m} (t) \right)$, $t>0$.
\end{te}
\begin{proof}
The Fourier-Laplace transform of \eqref{chepalle} can be easily derived as in Theorem \ref{teorema principe} and reads
\begin{align}
	\widehat{\widetilde{\mathfrak{w}}}_{\nu_1, \cdots, \nu_m}^{\beta, r} (\bm{\xi}, \mu) \, = \,  \frac{\sum_{j=1}^m \lambda_j \mu^{\nu_j^r-1}}{\sum_{j=1}^m \lambda_j \mu^{\nu_j^r}  + c^2 \left\| \bm{\xi} \right\|^{2\beta}}.
	\label{319}
\end{align}
By considering the law of the process $\bm{S}_{n}^{2\beta} \left( c^2 \mathfrak{L}^{\nu_1, \cdots, \nu_m}_r (t) \right)$ we have that
\begin{align*}
& \int_0^\infty dt \, e^{-\mu t} \mathbb{E}e^{i \bm{\xi} \cdot \bm{S}_{n}^{2\beta} \left( c^2 \mathfrak{L}_r^{\nu_1, \cdots, \nu_m} (t) \right)} \notag \\
 =  \, & \int_{\mathbb{R}^n} d\bm{x} \, e^{i \bm{\xi} \cdot \bm{x}} \int_0^\infty dt \, e^{-\mu t} \int_0^\infty ds \, v_\beta \left( \bm{x}, c^2s \right) \, \mathfrak{l}_{\nu_1, \cdots, \nu_m}^{r} \left( s, t \right) \notag \\
	= \, & \int_0^\infty ds \, e^{-sc^2 \left\| \bm{\xi} \right\|^{2\beta}} \int_0^\infty dt \, e^{-\mu t} \left[- \frac{\partial}{\partial s} \int_0^t \mathfrak{h}_{\nu_1, \cdots, \nu_m}^{r} (z, s) \, dz \right]  \notag \\
	= \, & - \frac{1}{\mu}\int_0^\infty ds \, e^{-s c^2 \left\| \bm{\xi} \right\|^{2\beta}} \left(  \frac{\partial}{\partial s} \widetilde{\mathfrak{h}}_{\nu_1, \cdots, \nu_m}^{r} (\mu, s) \right) \notag \\
= \, & - \frac{1}{\mu}\int_0^\infty ds \, e^{-s c^2 \left\| \bm{\xi} \right\|^{2\beta}} \left(  \frac{\partial}{\partial s} \mathbb{E}e^{-\mu \sum_{j=1}^m \lambda_j^{\frac{1}{\nu_j}} \, _1H^{\nu_j} \left( \, _2H^{\nu_j} \left( \, _3H^{\nu_j} \left( \cdots \, _{r}H^{\nu_j} (t) \cdots \right) \right) \right) } \right) \notag \\
\stackrel{\eqref{557}}{=} \, & - \frac{1}{\mu}\int_0^\infty ds \, e^{-s c^2 \left\| \bm{\xi} \right\|^{2\beta}} \left(  \frac{\partial}{\partial s} e^{-s \sum_{j=1}^m \lambda_j \mu^{\nu_j^r}} \right) \notag \\
= \, & \sum_{j=1}^m \lambda_j \mu^{\nu_j^r-1} \int_0^\infty ds \, e^{-s c^2 \left\| \bm{\xi} \right\|^{2\beta}-s \sum_{j=1}^m \lambda_j \mu^{\nu_j^r}} \, = \, \frac{\sum_{j=1}^m \lambda_j \mu^{\nu_j^r-1}}{ \sum_{j=1}^m \lambda_j \mu^{\nu_j^r} + c^2 \left\| \bm{\xi} \right\|^{2\beta}}
	\label{}
\end{align*}
which coincides with \eqref{319}.
\end{proof}
We now consider the limiting case for $r \to \infty$ where the iteration of the process $\bm{S}_n^{2\beta} \left( c^2 \mathfrak{L}^{\nu_1, \cdots, \nu_m}_r (t) \right)$, $t>0$, is infinitely extended. In the next theorem we have that the limiting law of
\begin{equation*}
\lim_{r \to \infty} \bm{S}_n^{2\beta} \left( c^2 \mathfrak{L}^{\nu_1, \cdots, \nu_m}_r (t) \right), \qquad t>0,
\end{equation*}
is, for $\beta = 1$, a generalization to $\mathbb{R}^n$ of the Gauss-Laplace probability density. This result  represents an extension to the $n$-dimensional case of the infinitely iterated Brownian motion (see \cite{orsann}).
\begin{te}
\label{teoremaconvergenza}
The distribution of the limiting process
\begin{equation*}
\lim_{r \to \infty} \bm{B}_n \left( c^2 \mathfrak{L}^{\nu_1, \cdots, \nu_m}_r (t) \right) \stackrel{\textrm{law}}{=} X_{m,n}
\end{equation*}
does not depend on $t$ and reads
\begin{align}
 \mathfrak{w}_m (\bm{x})  =  \frac{\Pr \left\lbrace X_{m,n} \in d\bm{x} \right\rbrace}{d\bm{x}}  =   \frac{1}{(2\pi)^{\frac{n}{2}}} \left(  \frac{\sqrt{\sum_{j=1}^m \lambda_j}}{c} \right)^{\frac{n+2}{2}}   \left\| \bm{x} \right\|^{-\frac{n-2}{2}} K_{\frac{n-2}{2}} \left( \frac{\sqrt{\sum_{j=1}^m \lambda_j}}{c} \left\| \bm{x} \right\|  \right).
\label{density39}
\end{align}
The density \eqref{density39} solves the equation
\begin{equation*}
\left( \sum_{j=1}^m \lambda_j \right) \mathfrak{w}_m (x_1, \cdots, x_n) \, = \, c^2 \sum_{j=1}^n \frac{\partial^2}{\partial x_j^2} \mathfrak{w}_m (x_1, \cdots, x_n),
\end{equation*}
which is obtained from \eqref{chepalle} by letting $r \to \infty$.
\end{te}
\begin{proof}
By assuming
\begin{align*}
A = \frac{1}{(2\pi)^{\frac{n}{2}}} \left( \frac{\sqrt{\sum_{j=1}^m \lambda_j}}{c} \right)^{\frac{n+2}{2}}, \qquad B  = \frac{\sqrt{\sum_{j=1}^m \lambda_j}}{c},
\end{align*}
the density
\begin{equation*}
\mathfrak{w}_m (\bm{x}) \, = \, A  \left(  \sum_{j=1}^n x_j^2 \right)^{-\frac{n-2}{4}} K_{\frac{n-2}{2}} \left( B \sqrt{\sum_{j=1}^n x_j^2}  \right)
\end{equation*}
has first-order derivative which reads
\begin{align}
& \frac{\partial}{\partial x_j} \mathfrak{w}_m (\bm{x}) \, = \notag \\
= \, &  AB \, \frac{x_j}{\left( \sum_{j=1}^n x_j^2 \right)^{\frac{n}{4}}}  K_{\frac{n-2}{2}}^\prime \left( B \left( x_1^2 + \cdots + x_n^2 \right)^{\frac{1}{2}} \right) -A  \frac{ \left( \frac{n-2}{2} \right) \, x_j \; K_{\frac{n-2}{2}} \left( B \left( x_1^2 + \cdots + x_n^2 \right)^{\frac{1}{2}} \right)}{\left( \sum_{j=1}^n x_j^2 \right)^{\frac{n}{4}+\frac{1}{2}}}   \notag \\
= \, & -AB x_j \left( \sum_{j=1}^n x_j^2  \right)^{-\frac{n}{4}} K_{\frac{n}{2}} \left( B \left( x_1^2 + \cdots + x_n^2 \right)^{\frac{1}{2}} \right).
\label{derivataprima}
\end{align}
In the last step we applied the relationship
\begin{equation}
\frac{d}{dz} K_\nu (z) = \frac{\nu}{z} K_\nu (z) - K_{\nu+1} (z)
\label{formulaleb}
\end{equation}
of Lebedev \cite{Leb}, page 110. The second-order derivative now becomes
\begin{align}
& \frac{\partial^2}{\partial x_j^2} \mathfrak{w}_m \left( x_1, \cdots, x_n \right) \, = \notag \\
= \, & -AB \frac{1}{\left( \sum_{j=1}^n x_j^2 \right)^{\frac{n}{4}}}  K_{\frac{n}{2}} \left( B \left( x_1^2 + \cdots + x_n^2 \right)^{\frac{1}{2}} \right)  + \frac{n}{2} AB \frac{x_j^2 \; K_{\frac{n}{2}} \left( B \left( x_1^2 + \cdots + x_n^2 \right)^{\frac{1}{2}} \right)}{\left( \sum_{j=1}^n x_j^2 \right)^{\frac{n}{4}+1}}     \notag \\
& - AB^2 \frac{x_j^2}{\left( \sum_{j=1}^n x_j^2 \right)^{\frac{n}{4}+\frac{1}{2}}}   K^\prime_{\frac{n}{2}} \left( B \left( x_1^2 + \cdots + x_n^2 \right)^{\frac{1}{2}} \right)  \notag \\
 = \, & AB^2  \frac{x_j^2}{\left( \sum_{j=1}^n x_j^2 \right)^{\frac{n}{4}+\frac{1}{2}}}  K_{\frac{n}{2}+1} \left( B \left( x_1^2 + \cdots + x_n^2 \right)^{\frac{1}{2}} \right) -AB \frac{K_{\frac{n}{2}} \left( B \left( x_1^2 + \cdots + x_n^2 \right)^{\frac{1}{2}} \right)}{\left( \sum_{j=1}^n x_j^2 \right)^{\frac{n}{4}}}.
\label{derivataseconda}
\end{align}
By considering the relationship
\begin{equation*}
K_{\nu+1} (z) = K_{\nu-1} (z) + 2 \frac{\nu}{z} K_\nu (z)
\end{equation*}
of \cite{Leb}, page 110, the derivative \eqref{derivataseconda} takes the form
\begin{align*}
 \frac{\partial^2}{\partial x_j^2} \mathfrak{w}_m \left( x_1, \cdots, x_n \right) \, = \,  AB^2 x_j^2 \left( \sum_{j=1}^n x_j^2 \right)^{-\frac{n}{4}-\frac{1}{2}} K_{\frac{n}{2}-1} \left( B \left( x_1^2+ \cdots + x_n^2 \right) \right).
\end{align*}
The Laplacian of $\mathfrak{w}_m (x_1, \cdots, x_m)$ therefore becomes
\begin{align*}
\sum_{j=1}^n \frac{\partial^2}{\partial x_j^2} \mathfrak{w}_m \left( \bm{x} \right) \, = \, AB^2 \left( \sum_{j=1}^n x_j^2 \right)^{-\frac{n+2}{4}} K_{\frac{n-2}{2}} \left( B \left( \sum_{j=1}^n x_j^2 \right)^{\frac{1}{2}} \right)
\end{align*}
and thus taking $A$ and $B$ explicitly we obtain the desired result
\begin{align*}
c^2 \sum_{j=1}^n \frac{\partial^2}{\partial x_j^2} \mathfrak{w}_m  (x_1, \cdots, x_n)  \, = \, \sum_{j=1}^m \lambda_j \mathfrak{w}_m  (x_1, \cdots, x_n).
\end{align*}
\end{proof}
\begin{os} \normalfont
For $r \to \infty$ the Fourier-Laplace transform \eqref{319} becomes
\begin{align*}
	\widehat{\widetilde{\mathfrak{w}}}_m^\beta (\bm{\xi}, \mu) \, = \, \frac{1}{\mu} \frac{\sum_{j=1}^m \lambda_j}{\sum_{j=1}^m \lambda_j + c^2 \left\| \bm{\xi} \right\|^{2\beta}},
	\label{}
\end{align*}
and thus the Fourier transform takes the form
\begin{equation}
\widehat{\mathfrak{w}}_m^\beta \left( \bm{\xi}\right) \, = \,  \frac{\sum_{j=1}^m \lambda_j}{\sum_{j=1}^m \lambda_j + c^2 \left\| \bm{\xi} \right\|^{2\beta}}.
\label{booh}
\end{equation}
The inversion of the Fourier transform \eqref{booh} can be carried out by means of the hyperspherical coordinates. Thus we have that
\begin{align*}
	   & \mathfrak{w}_m^\beta \left( \bm{x} \right) \, =  \, \frac{1}{(2\pi)^n} \int_{\mathbb{R}^n} e^{-i \bm{\xi} \cdot \bm{x}} \frac{\sum_{j=1}^m \lambda_j}{\sum_{j=1}^m \lambda_j + c^2 \left\| \bm{\xi} \right\|^{2\beta}} d\bm{\xi} \notag \\
	= \, & \frac{1}{(2\pi)^n} \int_0^\infty d\rho \, \rho^{n-1} \frac{\sum_{j=1}^m \lambda_j}{\sum_{j=1}^m \lambda_j + c^2 \rho^{2\beta}}  \int_0^\pi d\theta_1 \int_0^\pi d\theta_2  \cdots \int_0^\pi d\theta_{n-2} \notag \\
& \int_0^{2\pi} d\phi \, e^{-i \rho \left[ x_1 \sin \theta_1 \sin \theta_2 \cdots \sin \theta_{n-2} \sin \phi + x_2 \sin \theta_1 \sin \theta_2 \cdots \sin \theta_{n-2} \cos \phi\right]} \notag \\
	& e^{-i \rho\left[ x_3 \sin \theta_1 \sin \theta_2 \cdots \sin \theta_{n-3} \cos \theta_{n-2} + \;  \cdots \;  + x_{n-1} \sin \theta_1 \cos \theta_2 + x_n \cos \theta_1 \right]}  \sin^{n-2} \theta_1 \; \cdots \;  \sin \theta_{n-2}  \notag \\
	= \, & \frac{1}{(2\pi)^{n-1}} \int_0^\infty  \rho^{n-1} \frac{\sum_{j=1}^m \lambda_j}{\sum_{j=1}^m \lambda_j + c^2 \rho^{2\beta}} d\rho \int_0^\pi d\theta_1 \int_0^\pi d\theta_2  \cdots \int_0^\pi d\theta_{n-2} \, \sin^{n-2} \theta_1  \cdots \sin \theta_{n-2} \notag \\ 
	& e^{-i \rho\left[ x_3 \sin \theta_1 \sin \theta_2 \cdots \sin \theta_{n-3} \cos \theta_{n-2} + \;  \cdots \;  + x_{n-1} \sin \theta_1 \cos \theta_2 + x_n \cos \theta_1 \right]} \notag \\
	&  J_0 \left( \rho \sqrt{x_1^2 + x_2^2} \sin \theta_1 \sin \theta_2 \cdots \sin \theta_{n-2} \right)
	\label{}
\end{align*}
We now evaluate the integrals with respect to $\theta_j$ by means of formula 6.688 page 727 of Gradshteyn and Ryzhik \cite{G-R}, which reads
\begin{equation*}
\int_0^{\frac{\pi}{2}}  \sin^{\nu +1} x  \cos \left( \beta \cos x \right) J_\nu \left( \alpha \sin x  \right)  dx  =  \sqrt{\frac{\pi}{2}}  \frac{\alpha^\nu}{\left( \alpha^2 + \beta^2 \right)^{\frac{\nu}{2}+\frac{1}{4}}}  J_{\nu+\frac{1}{2}} \left( \sqrt{\alpha^2 + \beta^2} \right).
\end{equation*}
valid for $\Re (\nu) > -1$. We start with the integral with respect to $\theta_{n-2}$
\begin{align*}
& \int_0^\pi d\theta_{n-2} \, e^{-i \rho x_3 \sin \theta_1 \cdots \sin \theta_{n-3} \cos \theta_{n-2}} \sin \theta_{n-2} \, J_0 \left( \rho \sqrt{x_1^2+x_2^2} \sin \theta_1 \cdots \sin \theta_{n-2} \right) \notag \\
= \, & 2 \int_0^{\frac{\pi}{2}} d\theta_{n-2} \cos \left( \rho x_3 \sin \theta_1  \cdots \sin \theta_{n-3} \cos \theta_{n-2} \right) \sin \theta_{n-2}  J_0 \left( \rho \sqrt{x_1^2+x_2^2} \, \sin \theta_1  \cdots \sin \theta_{n-2} \right) \notag \\
= \, & \sqrt{2\pi} \left( \rho \sin \theta_1 \cdots \sin \theta_{n-3} \sqrt{x_1^2+x_2^2+x_3^2} \right)^{-\frac{1}{2}}  J_{\frac{1}{2}} \left( \rho \sqrt{x_1^2+x_2^2+x_3^2} \, \sin \theta_1  \cdots \sin \theta_{n-3} \right)
\end{align*}
and thus the integral with respect to $\theta_{n-3}$ becomes
\begin{align*}
&\sqrt{2\pi} \int_0^\pi d\theta_{n-3} e^{-i\rho x_4 \sin \theta_1 \cdots \sin \theta_{n-4} \cos \theta_{n-3}} \sin^2 \theta_{n-3}  \notag \\
& \left( \rho \sin \theta_1 \cdots \sin \theta_{n-3} \sqrt{x_1^2+x_2^2+x_3^2} \right)^{-\frac{1}{2}} J_{\frac{1}{2}} \left( \rho \sqrt{x_1^2+x_2^2+x_3^2} \sin \theta_1 \cdots \sin \theta_{n-3} \right) \notag \\
= \, & 2 \sqrt{2\pi} \left( \rho \sin \theta_{n-1} \cdots \sin \theta_{n-4} \sqrt{x_1^2+x_2^2+x_3^2} \right)^{-\frac{1}{2}} \int_0^{\frac{\pi}{2}} d\theta_{n-3} \, \sin^{\frac{3}{2}} \theta_{n-3} \notag \\
& \cos \left( \rho x_4 \sin \theta_1 \cdots \sin \theta_{n-4} \cos \theta_{n-3} \right) \, J_{\frac{1}{2}} \left( \rho \sqrt{x_1^2+x_2^2+x_3^2} \, \sin \theta_1 \cdots \sin \theta_{n-3}  \right) \notag \\
= \, & \left( \sqrt{2\pi} \right)^2 \left( \rho^2 \sin^2 \theta_1 \cdots \sin^2 \theta_{n-4} \left( x_1^2+x_2^2+x_3^2+x_4^2 \right) \right)^{-\frac{1}{2}} \notag \\
& J_1 \left( \rho \sqrt{x_1^2+x_2^2+x_3^2+x_4^2} \sin \theta_1 \sin \theta_2 \cdots \sin \theta_{n-4} \right).
\end{align*}
After $n-2$ integrations we arrive at the integral with respect to $\rho$ which reads
\begin{equation}
\mathfrak{w}_m^\beta \left( \bm{x} \right) \, = \, (2\pi)^{-\frac{n}{2}} \int_0^\infty d\rho \, \frac{\sum_{j=1}^m \lambda_j}{\sum_{j=1}^m \lambda_j + c^2 \rho^{2\beta}} J_{\frac{n-2}{2}} \left( \rho \sqrt{\sum_{i=j}^n x_j^2} \right) \, \frac{\rho^{\frac{n}{2}}}{\left( \sqrt{\sum_{i=j}^n x_j^2} \right)^{\frac{n-2}{2}}}
\label{remark}
\end{equation}
which, for $ \beta = 1$ and after the change of variable
 $\rho^2c^2=y^2$, becomes
\begin{align*}
 \mathfrak{w}_m \left( \bm{x} \right) \, = \, & (2\pi)^{-\frac{n}{2}} \frac{1}{c} \int_0^\infty dy \, \frac{\sum_{j=1}^m \lambda_j}{\sum_{j=1}^m \lambda_j + y^2} J_{\frac{n-2}{2}} \left( \frac{y}{c} \sqrt{\sum_{j=1}^n x_j^2} \right) \,  \frac{\left(  \frac{y}{c} \right)^{\frac{n}{2}}}{\left( \sqrt{\sum_{j=1}^n x_j^2} \right)^{\frac{n-2}{2}}}  \notag  \\
\stackrel{\textrm{for } n<5}{=} \, & \frac{1}{(2\pi)^{\frac{n}{2}}} \left(  \frac{\sqrt{\sum_{j=1}^m \lambda_j}}{c} \right)^{\frac{n+2}{2}}  \left( \sqrt{\sum_{j=1}^n x_j^2} \right)^{-\frac{n-2}{2}} K_{\frac{n-2}{2}} \left( \frac{\sqrt{\sum_{j=1}^m \lambda_j}}{c} \sqrt{\sum_{j=1}^n x_j^2}  \right) \notag \\
= \, & \frac{1}{(2\pi)^{\frac{n}{2}}} \left(  \frac{\sqrt{\sum_{j=1}^m \lambda_j}}{c} \right)^{\frac{n+2}{2}}   \left\| \bm{x} \right\|^{-\frac{n-2}{2}} K_{\frac{n-2}{2}} \left( \frac{\sqrt{\sum_{j=1}^m \lambda_j}}{c} \left\| \bm{x} \right\|  \right),
\end{align*}
where we used formula 6.566 page 679 of \cite{G-R}, which reads
\begin{equation*}
\int_0^\infty dx \, x^{\nu+1} J_{\nu} (ax) \frac{1}{x^2+b^2} \, = \, b^\nu K_\nu (ab), \qquad a>0, \, \Re (b) >0, \, -1 < \Re (\nu) < \frac{3}{2}.
\end{equation*}
\end{os}
\begin{os}  \normalfont
We can check that \eqref{density39} for all $n \in \mathbb{N}$ is a true probability density.
\begin{align*}
 \int_{\mathbb{R}^n} \mathfrak{w}_m (\bm{x}) \, d\bm{x} \, = \, & \frac{\textrm{area} \left( S_1^n \right)}{(2\pi)^{\frac{n}{2}}} \left( \frac{\sqrt{\sum_{j=1}^m \lambda_j}}{c} \right)^{\frac{n+2}{2}} 
 \int_0^\infty \rho^{n-1-\frac{n-2}{2}} K_{\frac{n}{2}-1} \left(  \rho \frac{\sqrt{\sum_{j=1}^m \lambda_j}}{c}  \right) d\rho \notag \\
 = \, & \frac{(2\pi)^{\frac{n}{2}}}{\Gamma \left( \frac{n}{2} \right)} \frac{1}{(2\pi)^{\frac{n}{2}}} \left(   \frac{\sqrt{\sum_{j=1}^m \lambda_j}}{c}  \right)^{\frac{n+2}{2}} \int_0^\infty \rho^{\frac{n}{2}} K_{\frac{n}{2}-1} \left(  \rho \frac{\sqrt{\sum_{j=1}^m \lambda_j}}{c}  \right) d\rho \, = \, 1
\end{align*}
in force of formula 6.561(16) of \cite{G-R} page 676
\begin{equation}
\int_0^\infty x^\mu K_\nu \left( ax \right) \, = \, 2^{\mu-1} a^{-\mu -1} \Gamma \left( \frac{1+\mu+\nu}{2} \right) \Gamma \left( \frac{1+\mu -\nu}{2} \right),
\label{momenti}
\end{equation}
valid for $\Re \left( \mu+1\pm \nu \right) >0$ and $\Re (a) >0$. The non-negativity of \eqref{density39} is shown by the following integral representation
\begin{equation*}
K_\nu (z) \, = \, \int_0^\infty e^{-z \cosh t} \cosh \nu t \, dt
\end{equation*}
valid for $\left| \arg (z) \right| < \frac{\pi}{2}$ (see \cite{G-R} page 917 formula 8.432).

By considering that
\begin{equation}
K_{-\frac{1}{2}} (z) \, = \, K_{\frac{1}{2}} (z) \, = \, \sqrt{\frac{\pi}{2z}} e^{-z},
\label{kappaunmezzo}
\end{equation}
from \eqref{density39} we derive the following probability density for $\bm{x} \in \mathbb{R}^3$,
\begin{align*}
& \mathfrak{w}_m (x_1, x_2, x_3) \, = \notag \\
= \, & \frac{\sum_{j=1}^m \lambda_j}{(2c)^2 \pi \sum_{j=1}^n x_j^2} e^{-\frac{\sqrt{\sum_{j=1}^m \lambda_j}}{c} \sqrt{\sum_{j=1}^n x_j^2} } \, = \, \frac{\sum_{j=1}^m \lambda_j}{(2c)^2 \pi \left\| \bm{x} \right\|^2} e^{-\frac{\sqrt{\sum_{j=1}^m \lambda_j}}{c} \left\| \bm{x} \right\| }
\end{align*}
In the two dimensional case the distribution \eqref{density39} has a simple structure which reads
\begin{equation*}
\mathfrak{w}_m (x_1, x_2) \, = \, \frac{1}{2\pi} \frac{\sum_{j=1}^m \lambda_j}{c^2} K_0 \left( \frac{\sqrt{\sum_{j=1}^m \lambda_j}}{c} \left\| \bm{x} \right\| \right).
\end{equation*}
In view of \eqref{kappaunmezzo} it is also easy to show that the distribution \eqref{density39} coincides for $n=1$ with the classical Gauss-Laplace distribution. We have that for $n=1$ \eqref{density39} becomes
\begin{align}
\mathfrak{w}_m (x) \, = \, & \frac{1}{\sqrt{2\pi}} \left( \frac{\sqrt{\sum_{j=1}^m \lambda_j}}{c} \right)^{\frac{3}{2}} \, \sqrt{\left| x \right|} \sqrt{\frac{\pi c}{2\sqrt{\sum_{j=1}^m \lambda_j} \left| x \right|}} e^{-\frac{\sqrt{\sum_{j=1}^m \lambda_j}}{c} \left| x \right|} \notag \\
 = \, &  \frac{\sqrt{\sum_{j=1}^m \lambda_j}}{2c} e^{-\frac{\sqrt{\sum_{j=1}^m  \lambda_j}}{c}|x|}
\label{nugualea1}
\end{align}
Furthermore, for $\lambda_1 = 1$, $\lambda_2 = 2\lambda$, $\lambda >0$ and $\lambda_j = 0$ for $j = 3, \cdots, m$, we note that \eqref{nugualea1} coincides with formula (3.18) of \cite{orsann}.
\end{os}
\begin{os}  \normalfont
By considering the iterated random walk
\begin{equation*}
Y_n (k) = S_1 \left( S_2 \left( \cdots \left( S_n(k) \right)  \cdots \right) \right), \qquad k \in \mathbb{N},
\end{equation*}
with $S_j$, $j = 1, \cdots, n$, independent random walks, \cite{turbo} has shown that for $n \to \infty$, $Y_n(k)$ converges to a stationary r.v. (independent from $k$) which possesses Gauss-Laplace distribution, in agreement with result \eqref{density39} of the present work and with (3.12) of \cite{orsann}.
\end{os}

 \vspace{0.8cm}

	\end{document}